\documentclass[12pt]{amsart}
\usepackage{a4wide}
\usepackage{enumerate}
\usepackage{graphicx}
\usepackage{color}
\usepackage{amsmath}
\usepackage{scrextend} 
\allowdisplaybreaks

\let\pa\partial  
\let\na\nabla  
\let\eps\varepsilon  
\newcommand{\N}{{\mathbb N}}  
\newcommand{\R}{{\mathbb R}} 
\newcommand{\diver}{\operatorname{div}}

\newcommand{\DD}{{\mathcal D}}

\newtheorem{theorem}{Theorem}   
\newtheorem{lemma}[theorem]{Lemma}   
   
\newtheorem{remark}[theorem]{Remark}

\begin{document}  

\title[Degenerate cross-diffusion system for ion transport]{Analysis of a degenerate
parabolic cross-diffusion system for ion transport}

\author{Anita Gerstenmayer}
\address{Institute for Analysis and Scientific Computing, Vienna University of  
	Technology, Wiedner Hauptstra\ss e 8--10, 1040 Wien, Austria}
\email{anita.gerstenmayer@tuwien.ac.at} 
\author{Ansgar J\"ungel}
\address{Institute for Analysis and Scientific Computing, Vienna University of  
	Technology, Wiedner Hauptstra\ss e 8--10, 1040 Wien, Austria}
\email{juengel@tuwien.ac.at} 

\date{\today}

\thanks{The authors acknowledge partial support from   
the Austrian Science Fund (FWF), grants P27352, P30000, and W1245, and from
the Austrian-French project of the Austrian Exchange Service (\"OAD), grant
FR 04/2016} 

\begin{abstract}
A cross-diffusion system describing ion transport through biological
membranes or nanopores in a bounded domain with mixed Dirichlet-Neumann
boundary conditions is analyzed. The ion concentrations solve strongly coupled
diffusion equations with a drift term involving the electric potential
which is coupled to the concentrations through a Poisson equation.
The global-in-time existence of bounded weak solutions and
the uniqueness of weak solutions under moderate regularity assumptions are shown. 
The main difficulties of the analysis are the cross-diffusion terms 
and the degeneracy of the diffusion matrix, preventing the use of standard tools.
The proofs are based on the boundedness-by-entropy method, extended to
nonhomogeneous boundary conditions, and the uniqueness
technique of Gajewski. A finite-volume discretization in one space dimension 
illustrates the large-time behavior of the numerical solutions and shows that
the equilibration rates may be very small.
\end{abstract}

\keywords{Ion transport, existence of weak solutions, free energy, entropy method,
uniqueness of weak solutions, finite-volume approximation.}  
 
\subjclass[2000]{35K51, 35K65, 35Q92.}  

\maketitle


\section{Introduction}

The transport of ions through membranes or nanopores can be described on the
macroscopic level by the Poisson-Nernst-Planck equations, modeling ionic species
and an electro-neutral solvent in the self-consistent field \cite{Ner88}. 
The equations can be derived
in the mean-field limit from microscopic particle models \cite{NSSA04} and lead to
diffusion equations, satisfying Fick's law for the fluxes. 
This ansatz breaks down in narrow ion channels
if the finite size of the ions is taken into account.
Including size exclusion, the mean-field model, derived from an on-lattice model
in the diffusion limit \cite{BSW12,SLH09} or taking into account the combined effect
of the excess chemical potentials \cite{LiEi14}, leads to parabolic equations with
cross-diffusion terms. The aim of this paper is to analyze the 
cross-diffusion system of \cite{BSW12}.

\subsection{Model equations}

The evolution of the ion concentrations (volume fractions) $u_i$ and fluxes $J_i$
of the $i$th species is governed by the equations
\begin{equation}\label{1.eq}
  \pa_t u_i = \diver J_i, \quad J_i = D_i\big(u_0\na u_i - u_i\na u_0 
	+ u_0u_i(\beta z_i\na\Phi+\na W_i)\big)
\end{equation}
for $i=1,\ldots,n$, where $u_0=1-\sum_{i=1}^nu_i$ is the concentration
(volume fraction) of the
solvent. We have assumed that the molar masses are the same for all species.
Varying molar masses are considered in, e.g., \cite{ChJu15,DGM12} in the
context of the Maxwell-Stefan theory. The classical Nernst-Planck equations
are obtained after setting $u_0=1$ \cite{CBE92}. 
They can be also coupled with fluiddynamical equations; see, e.g., \cite{XSL14}.
Modified Nernst-Planck models without volume filling, but including 
cross-diffusion terms, were suggested and analyzed in \cite{HHLLL15,LiEi15}.
  
In equations \eqref{1.eq}, $D_i>0$ denotes the diffusion coefficients,
$\beta=q/(k_B\theta)>0$ is the inverse thermal voltage (or inverse thermal energy)
with the elementary charge $q$, the Boltzmann constant $k_B$, 
and the temperature $\theta$, 
$z_i\in\R$ is the valence of the $i$th species, and $W_i=W_i(x)$ is an 
external potential.
Note that Einstein's relation between the diffusivity $D_i$ and the
mobility $\mu_i=qD_i/(k_B\theta)=D_i\beta$ holds. 
The electrical potential $\Phi$ is determined by the Poisson equation
\begin{equation}\label{1.poi}
  -\lambda^2\Delta\Phi = \sum_{i=1}^nz_iu_i + f,
\end{equation}
where $\lambda>0$ is the (scaled) permittivity, $\sum_{i=1}^nz_iu_i$ is the total 
charge density, and $f=f(x)$ is a permanent charge density. 

Equations \eqref{1.eq}-\eqref{1.poi} are solved in the bounded domain 
$\Omega\subset\R^d$ ($d\ge 1$). Its boundary is supposed to consist of
an insulating part $\Gamma_N$, on which no-flux boundary conditions are prescribed, 
and the union $\Gamma_D$ of boundary contacts with external reservoirs, 
on which the concentrations are fixed. The electric potential is influenced by 
the voltage at $\Gamma_E$ between two electrodes, and we assume for simplicity that
$\Gamma_E=\Gamma_D$. This leads to the mixed Dirichlet-Neumann
boundary conditions
\begin{align}
  J_i\cdot\nu=0
	\mbox{ on }\Gamma_N, &\quad u_i=u_i^D\mbox{ on }\Gamma_D, \quad i=1,\ldots,n,
	\label{1.bc1} \\
	\na\Phi\cdot\nu=0\mbox{ on }\Gamma_N, &\quad \Phi=\Phi^D\mbox{ on }\Gamma_D.
	\label{1.bc2}
\end{align}
Finally, we prescribe the initial conditions
\begin{equation}\label{1.ic}
  u_i(\cdot,0) = u_i^0\quad\mbox{in }\Omega,\ i=1,\ldots,n.
\end{equation}

Equations \eqref{1.eq} can be written as the cross-diffusion system
\begin{equation}\label{1.eqA}
  \pa_t u_i = \diver\bigg(\sum_{j=1}^n A_{ij}(u)\na u_j + D_iu_0u_i\na F_i\bigg),
\end{equation}
where $F_i=\beta z_i\Phi+W_i$ is the effective potential and the diffusion
matrix $(A_{ij}(u))$ is defined by
$$
  A_{ii}(u) = D_iu_i, \quad A_{ij}(u) = D_i(u_0+u_i), \quad j\neq i.
$$
Mathematically, this system is strongly coupled with a nonsymmetric and
generally not positive semidefinite diffusion matrix such that the
existence of solutions to \eqref{1.eqA} is not trivial. A second difficulty
is the fact that a maximum principle is generally not available for 
cross-diffusion systems, and the proof of nonnegativity of 
$u_0=1-\sum_{i=1}^n u_i$ is unclear. The third problem arises due to the
degenerate structure hidden in the equations (see below for details).

For vanishing
potentials $F_i=0$, the global existence of bounded weak solutions to \eqref{1.eqA}
with no-flux boundary conditions has been shown in \cite{ZaJu17}, based on
the boundedness-by-entropy method \cite{Jue15,Jue16}. The existence of
weak solutions to the (easier) stationary problem was proved in \cite{BSW12}.
Related models were analyzed recently in \cite{BBRW17}.
No existence or uniqueness results for solutions to 
the full transient model \eqref{1.eq}-\eqref{1.ic}
seem to be available in the literature and in this paper, we fill this gap.
Compared to the works \cite{Jue15,ZaJu17}, the novelty here is the inclusion of
the electric potential and the mixed Dirichlet-Neumann boundary conditions,
which need to be treated in a careful way.

\subsection{Key idea of the analysis}

We extend the boundedness-by-entropy method \cite{Jue15}
to the case of nonconstant potentials and nonhomogeneous boundary conditions. 
The key observation, already stated
in \cite{BSW12}, is that \eqref{1.eq} possesses an entropy or gradient-flow
structure. The entropy or, more precisely,
free energy is given by 
\begin{align}
  & H(u) = \int_\Omega h(u)dx, \quad\mbox{where }u=(u_1,\ldots,u_n), \label{1.H} \\
	& h(u) = \sum_{i=0}^n\int_{u_i^D}^{u_i}\log\frac{s}{u_i^D}ds 
	+ \frac{\beta\lambda^2}{2}|\na(\Phi-\Phi^D)|^2
	+ \sum_{i=1}^n u_iW_i \nonumber
\end{align}
and $u_0^D=1-\sum_{i=1}^n u_i^D$. The free energy is bounded from below if
$u_i\in L^\infty(\Omega)$ and $W_i\in L^1(\Omega)$.
Equations \eqref{1.eqA} can be written as a formal gradient flow in the sense
\begin{equation}\label{1.eqB}
  \pa_t u_i = \diver\bigg(\sum_{j=1}^n B_{ij}\na w_j\bigg),
	\quad i=1,\ldots,n,
\end{equation}
where $B_{ii}=D_iu_0u_i$, $B_{ij}=0$ if $i\neq j$ provide a diagonal positive
semidefinite matrix $(B_{ij})$, and $w_j$ are the 
entropy variables, defined by
\begin{align}
  & \frac{\pa h}{\pa u_i} = w_i-w_i^D, \quad\mbox{where} \nonumber \\
	& w_i = \log\frac{u_i}{u_0}	+ \beta z_i\Phi + W_i, \quad 
	w_i^D = \log\frac{u_i^D}{u_0^D} + \beta z_i\Phi^D, \quad i=1,\ldots,n.
	\label{1.w}
\end{align}
We refer to Lemma \ref{lem.ev} below for the computation of $\pa h/\pa u_i$.
In thermodynamics $\pa h/\pa u_i$ is called the chemical potential of the
$i$th species. 
The advantage of formulation \eqref{1.eqB} is that the drift terms are eliminated and,
in this special case, the new diffusion matrix $(B_{ij})$ is diagonal.
Note that we have not included the boundary data into the formulation
\eqref{1.eqB}. In fact, the free energy is nonincreasing
along trajectories to \eqref{1.eq}-\eqref{1.ic} only if the boundary data are 
in equilibrium, i.e.\ if $\na w^D_i=0$. In the general case, the free energy
is bounded only; see \eqref{1.epi} below.

There is another important benefit of formulation \eqref{1.eqB}. 
Observing that the relation between $w=(w_1,\ldots,w_n)$ and $u=(u_1,\ldots,u_n)$ 
can be inverted explicitly according to
$$
  u_i = u_i(w) = \frac{\exp(w_i-\beta z_i\Phi-W_i)}{1+\sum_{j=1}^n
	\exp(w_j-\beta z_j\Phi-W_j)}, \quad i=1,\ldots,n,
$$
we see that, if $(w_1,\ldots,w_n,\Phi)$ is a solution to \eqref{1.poi} and
\eqref{1.eqB},
$$
  u_i(w) \in\DD := \bigg\{u=(u_1,\ldots,u_n)\in(0,1)^n:\sum_{i=1}^n u_i<1\bigg\}.
$$
This provides positive lower {\em and} upper bounds for the concentrations
$u_0,\ldots,u_n$ without the use of a maximum principle.

\subsection{Main results}

We prove (i) the global-in-time existence of bounded weak solutions, (ii) the
uniqueness of weak solutions under additional regularity assumptions, and (iii)
some numerical results on the large-time behavior of solutions in one space dimension. 
In the following, we detail these results. First, we specify the
technical assumptions.

\begin{labeling}{(A44)}
\item[(A1)] Domain: $\Omega\subset\R^d$ ($d\ge 1$) is a bounded domain with
$\pa\Omega=\Gamma_D\cup\Gamma_N\in C^{0,1}$, $\Gamma_D\cap\Gamma_N=\emptyset$,
$\Gamma_N$ is open in $\pa\Omega$, and $\mbox{meas}(\Gamma_D)>0$.

\item[(A2)] Parameters: $T>0$, $D_i$, $\beta>0$, and $z_i\in\R$,
$i=1,\ldots,n$.

\item[(A3)] Given functions: $f\in L^\infty(\Omega)$, $W_i\in H^1(\Omega)\cap
L^\infty(\Omega)$, 
and $W_i=0$ on $\Gamma_D$, $\na W_i\cdot\nu=0$ on $\Gamma_N$,
$i=1,\ldots,n$.

\item[(A4)] Initial and boundary data: $u_i^0\in L^\infty(\Omega)$,
$u_i^D\in H^1(\Omega)$, $u_i^0>0$, $u_i^D > 0$, $1-\sum_{i=1}^n u_i^0>0$, $1-\sum_{i=1}^n u_i^D>0$
in $\Omega$ for $i=1,\ldots,n$, 
and $\Phi^D\in H^1(\Omega)\cap L^\infty(\Omega)$ satisfies
$$
  -\lambda^2\Delta\Phi^D = f\quad\mbox{in }\Omega, \quad
	\na\Phi^D\cdot\nu = 0 \quad\mbox{on }\Gamma_N.
$$
\end{labeling}

Clearly, it is sufficient to define the functions $u_i^D$, $\Phi^D$
on $\Gamma_D$. By the extension property, they can be extended to $\Omega$,
and we assume in (A4) that the extension of $\Phi^D$ is done in a special way.
This extension is needed to be consistent with the definition of the free
energy (entropy) and the entropy variables; see Lemma \ref{lem.ev}.
We denote these extensions again by $u_i^D$, $\Phi^D$.
Furthermore, we introduce the space \cite{Tro87}
$$
  H^1_D(\Omega)=\{u\in H^1(\Omega):u=0\mbox{ on }\Gamma_D\}.
$$
The first result concerns the existence of bounded weak solutions.

\begin{theorem}[Global existence of weak solutions]\label{thm.ex}
Let Assumptions (A1)-(A4) hold. Then there exists a bounded weak solution
$u_1,\ldots,u_n:\Omega\times(0,T)\to\overline{\DD}$ to \eqref{1.eq}-\eqref{1.ic}
satisfying 
\begin{align*}
	& u_iu_0^{1/2},\ u_0^{1/2}\in L^2(0,T;H^1(\Omega)), \quad
  \pa_t u_i\in L^2(0,T;H^1_D(\Omega)'), \\
	& \Phi\in L^2(0,T;H^1(\Omega)), \quad i=1,\ldots,n,
\end{align*}
and the weak formulation
\begin{align}
  &\int_0^T\langle\pa_t u_i,\phi_i\rangle dt + D_i\int_0^T\int_\Omega
	u_0^{1/2}\big(\na(u_0^{1/2}u_i) - 3u_i\na u_0^{1/2}\big)\cdot\na\phi_i dxdt 
	\nonumber \\
	&\phantom{xx}{}
	+ D_i\int_0^T\int_\Omega u_i u_0(\beta z_i\na\Phi+\na W_i)\cdot\na\phi_i dxdt = 0, 
	\label{1.weak1} \\
	& \lambda^2\int_0^T\int_\Omega\na\Phi\cdot\na\theta dxdt
	= \int_0^T\int_\Omega\bigg(\sum_{i=1}^n z_iu_i + f\bigg)\theta dxdt, \label{1.weak2}
\end{align}
for all $\phi_i$, $\theta\in L^2(0,T;H^1_D(\Omega))$, $i=1,\ldots,n$.
The initial condition is satisfied in the sense of $H^1_D(\Omega)'$,
and the Dirichlet boundary conditions are given by 
$$
  u_0=u_0^D:=1-\sum_{i=1}^n u_i^D, \quad u_iu_0^{1/2}=u_i^D(u_0^D)^{1/2}
	\quad\mbox{on }\Gamma_D,\ i=1,\ldots,n,
$$ 
in the sense of traces in $L^2(\Gamma_D)$.
\end{theorem}

The proof is based on an approximation procedure, i.e.,
we prove first the existence of solutions $u_0^{(\tau)}$, $u_i^{(\tau)}$
to a regularized problem with approximation parameter $\tau>0$ and then
pass to the limit $\tau\to 0$. The estimates needed for the compactness
argument are coming from a discrete version of 
the entropy-production inequality (for simplicity, we omit the superindex $\tau$)
\begin{align}
  \frac{dH}{dt} 
	&= \int_\Omega\sum_{i=1}^n\pa_t u_i(w_i-w_i^D)dx
	= -\int_\Omega\sum_{i=1}^n D_iu_0u_i\na w_i\cdot\na(w_i-w_i^D)dx \nonumber \\
	&\le -\frac12\int_\Omega\sum_{i=1}^n D_iu_0u_i|\na w_i|^2 dx + C(w^D),
	\label{1.epi}
\end{align}
where the constant $C(w^D)>0$ depends on the $H^1(\Omega)$ norm of $w^D$. 
We show in \eqref{2.sqrt} below that
$$
  \sum_{i=1}^n u_iu_0\na\log\frac{u_i}{u_0} = 4u_0\sum_{i=1}^n|\na u_i^{1/2}|^2
	+ |\na u_0|^2 + 4|\na u_0^{1/2}|^2,
$$
which yields an $H^1(\Omega)$ estimate for $u_0^{1/2}$ but {\em not} for $u_i$
because of the factor $u_0\ge 0$. This reflects the degenerate nature of the
equations which is more apparent in the component-wise formulation
$\pa_t u_i=\diver(D_iu_0u_i\na w_i)$ (see \eqref{1.eqB}). 

To overcome this
degeneracy, we employ the technique developed in \cite{BDPS10,ZaJu17}.
We show that $(u_0^{(\tau)}u_i^{(\tau)})$ is bounded in $H^1(\Omega)$ and that
the (approximative) time derivative of $u_i^{(\tau)}$ is bounded in $H^1_D(\Omega)'$.
If $u_0^{(\tau)}$ was strictly positive, we could apply the Aubin-Lions lemma
to conclude strong convergence of (a subsequence of) $(u_i^{(\tau)})$ to some $u_i$
which solves \eqref{1.eq}. However, since $u_0^{(\tau)}$ may vanish in the
limit,
this lemma cannot be used. The idea is to compensate the lack of the 
gradient estimates for $u_i^{(\tau)}$ by exploiting the uniform estimates
for $u_0^{(\tau)}$. Then, by the ``degenerate'' Aubin-Lions lemma
(see, e.g., \cite[Appendix C]{Jue15}), (a subsequence of)
$(u_0^{(\tau)}u_i^{(\tau)})$ converges strongly to $u_0u_i$,
and $u_0$, $u_i$ solve \eqref{1.eq}. For details, see Section \ref{sec.ex}.

\begin{remark}\rm
1.~Theorem \ref{thm.ex} also holds when reaction terms $f_i(u)$ are introduced
on the right-hand side of \eqref{1.eq}. As in \cite{Jue15}, we need that
$f_i$ is continuous and $\sum_{i=1}^n f_i(u)(\pa h/\pa u_i)\le C(1+h(u))$
holds for some $C>0$ and all $u\in\DD$.

2.~The approximate solution satisfies a discrete version
of the entropy-production inequality; see \eqref{2.epi}.
As explained above, the sequence $(u_i^{(\tau)})$ may not converge strongly,
such that we are unable to perform the limit $\tau\to 0$ in \eqref{2.epi}. 
As a consequence, we cannot prove that the
free energy \eqref{1.H} is nonincreasing along trajectories of 
\eqref{1.eq}-\eqref{1.poi}, and the analysis of the large-time behavior seems
to be inaccessible. Therefore, we investigate the decay of $H(u)$ numerically;
see Section \ref{sec.num}.

3.~Since the Neumann boundary condition does not appear explicitly in the weak
formulation \eqref{1.weak1}-\eqref{1.weak2}, we do not need to make expressions like
$\na\Phi\cdot\nu=0$ on $\Gamma_N$ precise. We only mention along the way that terms 
like $\na\Phi\cdot\nu$ on $\Gamma_N$ have to be understood in the sense of
$H_{00}^{1/2}(\Gamma_N)'$ which is the dual space of $H_{00}^{1/2}(\Gamma_N)$
consisting of all functions $v$ on $\Gamma_N$ such that $v\in H^1_D(\Omega)$.
This space is larger than $H^{-1/2}(\Gamma_N)$. We refer to \cite[Chapter 18]{BaCa84} 
for details.
\qed
\end{remark}

The second result is the uniqueness of weak solutions.

\begin{theorem}[Uniqueness of weak solutions]\label{thm.unique}
Let Assumptions (A1)-(A4) hold, $\sum_{i=1}^n W_i\in L^\infty(0,T;W^{1,d}(\Omega))$, 
and let $D_i=1$ and $z_i=z\in\R$ for $i=1,\ldots,n$. Then there exists at most one
bounded weak solution to \eqref{1.eq}-\eqref{1.ic} in the class of functions
$u_i\in H^1(0,T;H_D^1(\Omega)')\cap L^2(0,T;H^1(\Omega))$, 
$\Phi\in L^\infty(0,T;W^{1,q}(\Omega))$ with $q>d$.
\end{theorem}

The proof is a combination of standard $L^2(\Omega)$-type estimates and
the entropy method of Gajewski \cite{Gaj94}. In fact, equations \eqref{1.eq}
partially decouple because of the assumptions $D_i=1$ and $z_i=z$.
Summing \eqref{1.eq} over $i=1,\ldots,n$, we find that $(u_0,\Phi)$ solves
\begin{equation}\label{1.aux}
  \pa_t u_0 = \diver\big(\na u_0 - u_0(1-u_0)(\beta z\na\Phi+\na W)\big), \quad
	-\lambda^2\Delta\Phi = z(1-u_0)+f(x),
\end{equation}
where $W=\sum_{i=1}^n W_i$.
The uniqueness of solutions is shown by taking two solutions $(u_0,\Phi)$ and
$(v_0,\Psi)$ and using $u_0-v_0$ as a test function in the first equation of
\eqref{1.aux}. Then, with the Gagliardo-Nirenberg inequality and the
hypothesis $\na\Phi\in L^q(\Omega)$, we show that
$$
  \frac{d}{dt}\int_\Omega(u_0-v_0)(t)^2 dx \le C(\Phi)\int_\Omega(u_0-v_0)^2 dx,
$$
where $C(\Phi)>0$ depends on the $W^{1,q}(\Omega)$ norm of $\Phi$.
Hence, Gronwall's lemma yields $u_0=v_0$ and consequently, $\Phi=\Psi$.

The next step is to show, for given $u_0$ and $\Phi$, that $u_i$ is the unique
solution to \eqref{1.eq}. Since we cannot expect that $\na u_i\in L^q(\Omega)$,
$q>d$, for $d\ge 3$, we employ the technique of Gajewski \cite{Gaj94} which avoids
this regularity. The method seems to work only for linear mobilities $u_i$,
which is the reason why we cannot apply it to \eqref{1.aux}.
The idea is to introduce the semimetric
$$
  d(u,v) = \int_\Omega\sum_{i=1}^n\bigg(h(u_i)+h(v_i)-2h\bigg(\frac{u_i+v_i}{2}\bigg)
	\bigg)dx \ge 0,
$$
where $h(s)=s(\log s-1)+1$, and to show that $\pa_t d(u,v)\le 0$. Since
$d(u(0),v(0))=0$, this implies that $d(u(t),v(t))=0$ for $t>0$
and consequently, $u(t)=v(t)$. Since expressions like $\log u_i$ are undefined
when $u_i=0$, we need to regularize the semimetric. For details, we refer to
Section \ref{sec.unique}.

\begin{remark}\rm
1.~ The regularity $u_i\in L^2(0,T;H^1(\Omega))$
holds if $u_0$ is strictly positive. A standard idea for the proof is to 
employ $\min\{0,u_0-me^{-\lambda t}\}^p$ as a test function in the first equation
of \eqref{1.aux}, where
$\inf_{\Gamma_D}u_0^D\ge m>0$ and $\lambda>0$ is sufficiently large, 
and to pass after some estimations to the limit $p\to \infty$. 
We leave the details to the reader;
see, e.g., \cite{GlHu02} for a proof in a related situation.

2.~The regularity condition $\Phi(t)\in W^{1,q}(\Omega)$ with $q>d$ is satisfied if 
$d\le 3$, $\pa\Omega\in C^{1,1}$, and the Dirichlet and Neumann boundary do not meet,
$\Gamma_D\cap\overline{\Gamma}_N=\emptyset$ \cite[Theorem~3.29]{Tro87}.
It is also satisfied in up to three space dimensions if $\pa\Omega\in C^3$, 
$\overline{\Gamma}_D\cap\overline{\Gamma}_N\in C^3$, and $\Phi^D\in
W^{1-1/q,q}(\Gamma_D)$, $q>d$ \cite{Sha68}.
\qed
\end{remark}

The paper is organized as follows. The existence theorem is proved in
Section \ref{sec.ex}, while the uniqueness result is shown in Section 
\ref{sec.unique}. The numerical solution in one space dimension
and its large-time behavior is illustrated in Section \ref{sec.num}. 
The entropy variables $\pa h/\pa u_i$ are computed in the Appendix.


\section{Existence of solutions}\label{sec.ex}

We consider first the nonlinear Poisson equation
$$
  -\lambda^2\Delta\Phi = \sum_{i=1}^n z_i u_i(w,\Phi) + f, \quad
	u_i(w,\Phi) = \frac{\exp(w_i-\beta z_i\Phi-W_i)}{1 + \sum_{j=1}^n
	\exp(w_j-\beta z_j\Phi-W_j)}
$$
in $\Omega$
with the boundary conditions \eqref{1.bc2} for given $w_i\in L^\infty(\Omega)$. 
Then $(x,\Phi)\mapsto u_i(w(x),\Phi)$ is a bounded function with values in $(0,1)$
and a standard fixed-point argument shows that  
this problem has a weak solution $\Phi\in H^1(\Omega)$.
Since $\Phi\mapsto u_i(w,\Phi)$ is Lipschitz continuous, this solution is unique.
By the maximum principle and $f\in L^\infty(\Omega)$, 
we have $\Phi\in L^\infty(\Omega)$.
Note that $u(w(x),\Phi(x))\in\DD$ for $x\in\Omega$. Therefore, the following
estimate holds:
\begin{equation}\label{2.Phi}
  \|\Phi\|_{H^1(\Omega)} \le C(1+\|\Phi^D\|_{H^1(\Omega)}),
\end{equation}
where $C>0$ depends on $\lambda$, $z_i$, and $\|f\|_{L^2(\Omega)}$.

{\em Step 1: Solution to an approximate problem.} Let $T>0$, $N\in\N$, $\tau=T/N>0$,
and $m\in\N$ such that $m>d/2$. Then the embedding $H^m(\Omega)\hookrightarrow
L^\infty(\Omega)$ is compact. Let $v^{k-1}:=w^{k-1}-w^{D}\in
H^1_D(\Omega;\R^n)\cap L^\infty(\Omega;\R^n)$, 
$\Phi^{k-1}-\Phi^D\in H^1_D(\Omega)$ be given. 
If $k=1$, we set $v^0=h'(u^0)-w^D$ and let $\Phi^0$ be the weak solution 
to $-\lambda^2\Delta\Phi^0=\sum_{i=1}^n z_iu_i^0+f(x)$
in $\Omega$ with boundary conditions \eqref{1.bc2}. Our aim is to find 
$v^k\in H^1_D(\Omega;\R^n)\cap H^m(\Omega;\R^n)$, $\Phi^k-\Phi^D\in H^1_D(\Omega)$ 
such that
\begin{align}
  & \frac{1}{\tau}\int_\Omega\big(u(v^k+w^D,\Phi^k) - u(v^{k-1}+w^D,\Phi^{k-1})
	\big)\cdot\phi dx \nonumber \\
	&\phantom{xxxx}{}+ \int_\Omega\na\phi:B(v^k+w^D,\Phi^k)\na(v^k+w^D)dx 
	\nonumber \\
	&\phantom{xxxx}{}
	+ \eps\int_\Omega\bigg(\sum_{|\alpha|=m}D^\alpha v^k\cdot D^\alpha\phi
	+ v^k\cdot\phi\bigg)dx = 0, \label{2.tau1} \\
	& \lambda^2\int_\Omega\na\Phi^k\cdot\na\theta dx = \int_\Omega\bigg(\sum_{i=1}^n
	z_i u_i(v^k+w^D,\Phi^k)+f\bigg)\theta dx \label{2.tau2}
\end{align}
for all $\phi\in H^1_D(\Omega;\R^n)$ and $\theta\in H^1_D(\Omega)$.
Here, $\alpha=(\alpha_1,\ldots,\alpha_n)\in\N_0^n$ is a multi-index, 
$|\alpha|=\alpha_1+\cdots+\alpha_n$, $D^\alpha = \pa^{|\alpha|}/\pa x_1^{\alpha_1}
\cdots\pa x_n^{\alpha_n}$ is a partial derivative, and ``:'' denotes
the matrix product with summation over both indices. Since the matrix $B$ is
diagonal, we may write the second integral in \eqref{2.tau1} as
\begin{align*}
  \int_\Omega\na\phi & :B(v^k+w^D,\Phi^k)\na(v^k+w^D)dx \\
	&= \int_\Omega\sum_{i=1}^n D_iu_0(v^k+w^D,\Phi^k)u_i(v^k+w^D,\Phi^k)\na\phi_i
	\cdot\na(v_i^k+w_i^D)dx.
\end{align*}

\begin{lemma}[Existence of weak solutions to the time-discrete problem]
\label{lem.approx}
Let the assumptions of Theorem \ref{thm.ex} hold and let $w^D\in H^m(\Omega;\R^n)$. 
Then there exists a weak
solution $v^k=w^k-w^D\in H^1_D(\Omega;\R^n)\cap H^m(\Omega;\R^n)$, 
$\Phi^k-\Phi^D\in H^1_D(\Omega)$ to \eqref{2.tau1}-\eqref{2.tau2}, and
the following discrete entropy production inequality holds:
\begin{align}
  H(u^k) &+ \tau\int_\Omega\na (w^k-w^D):B(w^k,\Phi^k)\na w^k dx \nonumber \\
	&{}+ \eps\tau C_P\|w^k-w^D\|_{H^m(\Omega)}^2 \le H(u^{k-1}), 
	\label{2.epi}
\end{align}
where $H$ is defined in \eqref{1.H},
$u^k=u(w^k,\Phi^k)$, $u^{k-1}=u(w^{k-1},\Phi^{k-1})$, and $C_P>0$
is the constant of the generalized Poincar\'e inequality 
\cite[Chap.~II.1.4, Formula (1.39)]{Tem97}.
\end{lemma}

\begin{proof}
We employ the Leray-Schauder fixed-point theorem. For this, let
$y\in L^\infty(\Omega)$ and $\delta\in[0,1]$. Let $\Phi^k-\Phi^D\in H^1_D(\Omega)$
be the unique weak solution to the nonlinear problem
$$
  \lambda^2\int_\Omega\na\Phi^k\cdot\na\theta dx = \int_\Omega\bigg(\sum_{i=1}^n
	z_i u_i(y+w^D,\Phi^k)+f\bigg)\theta dx
$$
for $\theta\in H^1_D(\Omega)$. Since $y\in L^\infty(\Omega)$, the expression
$u_i(y+w^D,\Phi^k)$ is well-defined. Next, let
$X=H^1_D(\Omega;\R^n)\cap H^m(\Omega;\R^n)$ and consider the linear problem
\begin{equation}\label{2.LM}
  a(v,\phi)=F(\phi)\quad\mbox{for all }\phi\in X,
\end{equation}
where
\begin{align*}
  a(v,\phi) &= \int_\Omega\na\phi:B(y+w^D,\Phi^k)\na v dx
	+ \eps\int_\Omega\bigg(\sum_{|\alpha|=m}D^\alpha v\cdot D^\alpha\phi
	+ v\cdot\phi\bigg)dx, \\
	F(\phi) &= -\frac{\delta}{\tau}\int_\Omega\big(u(y+w^D,\Phi^k)
	- u(v^{k-1}+w^D,\Phi^{k-1})\big)\cdot\phi dx \\
	&\phantom{xx}{}- \delta\int_\Omega\na\phi:B(y+w^D,\Phi^k)\na w^D dx.
\end{align*}
The bilinear form $a$ and the linear form $F$ are continuous on $X$.
Furthermore, using the positive semi-definiteness of the matrix $B$ and
the generalized Poincar\'e inequality with constant $C_P>0$
\cite[Chap.~II.1.4, Formula (1.39)]{Tem97}, $a$ is coercive:
$$
  a(v,v) \ge \eps\int_\Omega\bigg(\sum_{|\alpha|=m}|D^\alpha v|^2 + |v|^2\bigg)dx
	\ge \eps C_P\|v\|_{H^m(\Omega)}^2.
$$
By the lemma of Lax-Milgram, there exists a unique
solution $v\in X\subset L^\infty(\Omega;\R^n)$ to \eqref{2.LM}. 
For later reference, we observe that, since the continuity constant for $F$ 
does not depend on $y$,
\begin{equation}\label{2.est}
  C(\eps)\|v\|_{H^m(\Omega)}^2 \le a(v,v) = F(v) \le C(\tau)\|v\|_{H^m(\Omega)},
\end{equation}
which gives a bound for $v$ in $H^m(\Omega)$ which is independent of $y$
and $\delta$. 

This defines the fixed-point operator
$S:L^\infty(\Omega;\R^n)\times[0,1]\to L^\infty(\Omega;\R^n)$, $S(y,\delta)=v$.
It clearly holds that $S(y,0)=0$ for all $y\in L^\infty(\Omega;\R^n)$.
The continuity of $S$ follows from standard arguments; see, e.g., the proof
of Lemma 5 in \cite{Jue15}. In view of the compact embedding $H^m(\Omega)
\hookrightarrow L^\infty(\Omega)$, $S$ is also compact. 
The uniform estimate for all fixed points of $S(\cdot,\delta)$ follows from
\eqref{2.est}. Thus, by the
Leray-Schauder fixed-point theorem, there exists $v^k\in X$ such that
$S(v^k,1)=v^k$ and $w^k:=v^k+w^D$, $\Phi^k$ solve \eqref{2.tau1}-\eqref{2.tau2}.

It remains to prove inequality \eqref{2.epi}. To this end, we employ
$\tau(w^k-w^D)\in X$ as a test function in the weak formulation of \eqref{2.tau1}.
Again, we set $u^k=u(w^k,\Phi^k)$, $u^{k-1}=u(w^{k-1},\Phi^{k-1})$. Then
\begin{align}
  \int_\Omega(u^k-u^{k-1})\cdot (w^k-w^D) dx
	&+ \tau\int_\Omega\na (w^k-w^D):B(w^k,\Phi^k)\na w^k \nonumber \\
	&{}+ \eps\tau\|w^k-w^D\|_{H^m(\Omega)}^2 \le 0. \label{2.aux}
\end{align}
To estimate the first integral, we take $x\in\Omega$ and set
$$
  g(u) = \sum_{i=0}^n\int_{u_i^D(x)}^{u_i}\log\frac{s}{u_i^D(x)}ds, \quad u\in\R^n,
$$
where we recall that $u_0^D=1-\sum_{i=1}^n u_i^D$. 
Then $(\pa g/\pa u_i)(u)=\log(u_i/u_i^D)-\log(u_0/u_0^D)$ and $g$ is convex. Hence,
$g(u^k)-g(u^{k-1})\le g'(u^k)\cdot(u^k-u^{k-1})$ or
$$
  \int_\Omega(g(u^k)-g(u^{k-1}))dx \le \int_\Omega\sum_{i=1}^n
	(u^k_i-u^{k-1}_i)\bigg(\log\frac{u_i^k}{u_0^k}-\log\frac{u_i^D}{u_0^D}\bigg)dx.
$$
Moreover, we infer from the Poisson equation that
\begin{align*}
  \beta\int_\Omega \sum_{i=1}^n & z_i(u^k_i-u^{k-1}_i)(\Phi^k-\Phi^D)dx
	= -\beta\lambda^2\int_\Omega\Delta(\Phi^k-\Phi^{k-1})(\Phi^k-\Phi^D)dx \\
	&= \beta\lambda^2\int_\Omega\na\big((\Phi^k-\Phi^D)-(\Phi^{k-1}-\Phi^D)\big)
	\cdot\na(\Phi^k-\Phi^D)dx \\
	&\ge \frac{\beta\lambda^2}{2}\int_\Omega|\na(\Phi^k-\Phi^D)|^2dx
	- \frac{\beta\lambda^2}{2}\int_\Omega|\na(\Phi^{k-1}-\Phi^D)|^2dx.	
\end{align*}
In view of these estimates, the first term in \eqref{2.aux} becomes
\begin{align*}
  \int_\Omega & (u^k-u^{k-1})\cdot (w^k-w^D) dx \\
	&= \int_\Omega\sum_{i=1}^n(u^k_i-u^{k-1}_i)\bigg(\log\frac{u_i^k}{u_0^k}
	- \log\frac{u_i^D}{u_0^D} + \beta z_i(\Phi^k-\Phi^D) + W_i\bigg)dx \\
  &\ge H(u^k) - H(u^{k-1}).
\end{align*}
We infer from \eqref{2.aux} that \eqref{2.epi} holds.
\end{proof}

{\em Step 2: A priori estimates.} Let $(w^k,\Phi^k)$ be a weak solution
to \eqref{2.tau1}-\eqref{2.tau2}. Then $u^k(x)=u(w^k(x),\Phi^k(x))\in\DD$
for $x\in\Omega$, so $(u^k)$ is bounded uniformly in $(\eps,\tau)$.

\begin{lemma}[A priori estimates]\label{lem.est}
The following estimates hold:
\begin{align}
  \|u_i^k\|_{L^\infty(\Omega)} + \eps\tau\sum_{j=1}^k\|w_i^j\|_{H^m(\Omega)}^2
	&\le C, \label{2.est1} \\
	\tau\sum_{j=1}^k\Big(\|(u_0^j)^{1/2}\|_{H^1(\Omega)}^2 + \|u_0^j\|_{H^1(\Omega)}^2 
	+ \|(u_0^j)^{1/2}\na(u_i^j)^{1/2}\|_{L^2(\Omega)}^2\Big) &\le C, \label{2.est2}
\end{align}
where here and in the following, $C>0$ is a generic constant independent of
$\eps$ and $\tau$. 
\end{lemma}

\begin{proof}
We need to estimate the second term on the left-hand side of the 
entropy-production inequality \eqref{2.epi}.
Since $B(w^k,\Phi^k)=\mbox{diag}(D_i u_i^ku_0^k)$, we obtain
\begin{align*}
  \na (w^k-w^D):B(w^k,\Phi^k)\na w^k
	&= \sum_{i=1}^n D_iu_i^ku_0^k|\na w_i^k|^2
	- \sum_{i=1}^n D_iu_i^ku_0^k\na w^D_i\cdot\na w_i^k \\
	&\ge \frac{D_{\rm min}}{2}\sum_{i=1}^n u_i^ku_0^k|\na w_i^k|^2
	- \frac{D_{\rm max}}{2}\sum_{i=1}^n|\na w_i^D|^2,
\end{align*}
where $D_{\rm min}=\min_{i=1,\ldots,n}D_i$, $D_{\rm max}=\max_{i=1,\ldots,n}D_i$,
and we used the fact that $0\le u_0^k,\,u_i^k\le 1$ in $\Omega$. Furthermore,
by definition \eqref{1.w} of the entropy variables,
$$
  |\na w_i^k|^2 
	= \bigg|\na\log\frac{u_i^k}{u_0^k} + \na(\beta z_i\Phi^k+W_i)\bigg|^2 
	\ge \frac12\bigg|\na\log\frac{u_i^k}{u_0^k}\bigg|^2
	- |\na(\beta z_i\Phi+W_i)|^2.
$$
Inserting these inequalities into \eqref{2.epi}, it follows that
\begin{align*}
  H(u^k) &+ \tau\frac{D_{\rm min}}{4}\int_\Omega\sum_{i=1}^n u_i^ku_0^k
	\bigg|\na\log\frac{u_i^k}{u_0^k}\bigg|^2 dx
	+ \eps\tau C_P\|w^k-w^D\|_{H^m(\Omega)}^2 \\
	&\le H(u^{k-1}) 
	+ \tau\frac{D_{\rm min}}{2}\int_\Omega\sum_{i=1}^n|\na(\beta z_i\Phi^k+W_i)|^2 dx 
	+ \tau\frac{D_{\rm max}}{2}\int_\Omega\sum_{i=1}^n|\na w_i^D|^2 dx.
\end{align*}
We resolve this recursion to find that
\begin{align*}
  H(u^k) &+ \tau\frac{D_{\rm min}}{4}\sum_{j=1}^k
	\int_\Omega\sum_{i=1}^n u_i^ju_0^j
	\bigg|\na\log\frac{u_i^j}{u_0^j}\bigg|^2 dx
	+ \eps\tau C_P\sum_{j=1}^k \|w^j-w^D\|_{H^m(\Omega)}^2 \\
	&\le H(u^0) 
	+ \tau\frac{D_{\rm min}}{2}\sum_{j=1}^k\int_\Omega\sum_{i=1}^n
	|\na(\beta z_i\Phi^j+W_i)|^2 dx 
	+ \tau k\frac{D_{\rm max}}{2}\int_\Omega\sum_{i=1}^n|\na w_i^D|^2 dx.
\end{align*}
Because of the $H^1(\Omega)$
estimate \eqref{2.Phi} for the electric potential and $\tau k\le T$, 
the right-hand side is uniformly bounded. Furthermore, using $\sum_{i=1}^n u_i^j
= 1-u_0^j$, 
\begin{align}
  \sum_{i=1}^n u_i^ju_0^j\bigg|\na\log\frac{u_i^j}{u_0^j}\bigg|^2
	&= 4u_0^j\sum_{i=1}^n|\na(u_i^j)^{1/2}|^2 - 2\na u_0^j\sum_{i=1}^n\na u_i^j
	+ 4|\na(u_0^j)^{1/2}|^2\sum_{i=1}^n u_i^j \nonumber \\
	&= 4u_0^j\sum_{i=1}^n|\na(u_i^j)^{1/2}|^2 + 2|\na u_0^j|^2
	+ 4|\na(u_0^j)^{1/2}|^2 - 4u_0^j|\na(u_0^j)^{1/2}|^2 \nonumber \\
	&= 4u_0^j\sum_{i=1}^n|\na(u_i^j)^{1/2}|^2 + |\na u_0^j|^2 + 4|\na(u_0^j)^{1/2}|^2.
	\label{2.sqrt}
\end{align}
This finishes the proof.
\end{proof}

{\em Step 3: Limit $\eps\to 0$.} We cannot perform the simultaneous limit
$(\eps,\tau)\to 0$ since we need an Aubin-Lions compactness result, 
which requires a uniform estimate for the discrete time
derivative of the concentrations in $H_D^1(\Omega;\R^n)'$ and not in the larger
space $X'=(H_D^1(\Omega;\R^n)\cap H^m(\Omega;\R^n))'$.
Let $k\in\{1,\ldots,N\}$ be fixed and let
$u_i^{(\eps)}=u_i^k$ and $\Phi^{(\eps)}=\Phi^k$ be a weak solution to
\eqref{2.tau1}-\eqref{2.tau2}. Set $u_0^{(\eps)}=1-\sum_{i=1}^n u_i^{(\eps)}$.
By Lemma \ref{lem.est}, there exist subsequences
of $(u_i^{(\eps)})$ and $(\Phi^{(\eps)})$, which are not relabeled, such that,
as $\eps\to 0$,
\begin{align}
  u_i^{(\eps)} \rightharpoonup^* u_i &\quad\mbox{weakly* in }L^\infty(\Omega),
	\label{2.eps0} \\
  (u_0^{(\eps)})^{1/2}\rightharpoonup u_0^{1/2}, 
	\quad \Phi^{(\eps)}\rightharpoonup\Phi
	&\quad\mbox{weakly in }H^1(\Omega), \ i=1,\ldots,n, \label{2.H1} \\
	u_0^{(\eps)}\to u_0, \quad \Phi^{(\eps)}\to \Phi
	&\quad\mbox{strongly in }L^2(\Omega), \label{2.eps1} \\
	\eps w_i^{(\eps)}\to 0 &\quad\mbox{strongly in }H^m(\Omega). \label{2.weps}
\end{align}
We have to pass to the limit $\eps\to 0$ in
\begin{align*}
  \int_\Omega & \na\phi:B(w^{(\eps)},\Phi^{(\eps)})\na w^{(\eps)} dx
	= \int_\Omega\sum_{i=1}^n D_i u_i^{(\eps)}u_0^{(\eps)}\na w_i^{(\eps)}
	\cdot\na\phi_i dx \\
	&= \int_\Omega\sum_{i=1}^n D_i\big(u_0^{(\eps)}\na u_i^{(\eps)}
	- u_i^{(\eps)}\na u_0^{(\eps)} + u_i^{(\eps)} u_0^{(\eps)}
	(\beta z_i\na\Phi^{(\eps)}+\na W_i)\big)	\cdot\na\phi_i dx \\
	&= \int_\Omega\sum_{i=1}^n D_i\Big((u_0^{(\eps)})^{1/2}\na\big(u_i^{(\eps)}
	(u_0^{(\eps)})^{1/2}\big) 
	- 3u_i^{(\eps)}(u_0^{(\eps)})^{1/2}\na(u_0^{(\eps)})^{1/2} \\
	&\phantom{xx}{}
	+ \beta z_i u_i^{(\eps)} u_0^{(\eps)}(\beta z_i\na\Phi^{(\eps)}+\na W_i)\Big)
	\cdot\na\phi_i dx.
\end{align*}

We claim that $u_i^{(\eps)}(u_0^{(\eps)})^{1/2}\rightharpoonup u_iu_0^{1/2}$
weakly in $H^1(\Omega)$. First, we observe that, because of \eqref{2.eps0} and
\eqref{2.eps1}, $u_i^{(\eps)}(u_0^{(\eps)})^{1/2}\rightharpoonup u_iu_0^{1/2}$
weakly in $L^2(\Omega)$. Then the claim follows from the bound
\begin{align}
  \big\|\na\big(u_i^{(\eps)}(u_0^{(\eps)})^{1/2}\big)\big\|_{L^2(\Omega)}
	&\le \|u_i^{(\eps)}\|_{L^\infty(\Omega)}\|\na(u_0^{(\eps)})^{1/2}\|_{L^2(\Omega)} 
	\nonumber \\
	&\phantom{xx}{}+ 2\|(u_i^{(\eps)})^{1/2}\|_{L^\infty(\Omega)}\|(u_0^{(\eps)})^{1/2}
	\na(u_i^{(\eps)})^{1/2}\|_{L^2(\Omega)} \le C, \label{2.na}
\end{align}
using \eqref{2.est2}. The compact embedding $H^1(\Omega)\hookrightarrow
L^2(\Omega)$ implies that
\begin{equation*}
  u_i^{(\eps)}(u_0^{(\eps)})^{1/2} \to u_iu_0^{1/2} \quad\mbox{strongly in }
	L^2(\Omega),
\end{equation*}
and by the $L^\infty(\Omega)$ bounds, this convergence also holds in $L^p(\Omega)$
for $p<\infty$. This shows that, taking into account \eqref{2.H1},
\begin{align*}
  (u_0^{(\eps)})^{1/2} & \na\big(u_i^{(\eps)}(u_0^{(\eps)})^{1/2}\big)
	- 3u_i^{(\eps)}(u_0^{(\eps)})^{1/2}\na(u_0^{(\eps)})^{1/2} \\
	&\rightharpoonup u_0^{1/2}\na(u_iu_0^{1/2}) - 3u_iu_0^{1/2}\na u_0^{1/2}
	\quad\mbox{weakly in }L^1(\Omega).
\end{align*}
In fact, since this sequence is bounded in $L^2(\Omega)$, the weak convergence
also holds in $L^2(\Omega)$.
Furthermore, by \eqref{2.eps1}, possibly for a subsequence,
$$
  u_i^{(\eps)}u_0^{(\eps)}\na\Phi^{(\eps)}
	\rightharpoonup u_iu_0\na\Phi \quad\mbox{weakly in }L^1(\Omega),
$$
and this convergence holds also in $L^2(\Omega)$. 

Then, performing the limit $\eps\to 0$ in \eqref{2.tau1}-\eqref{2.tau2} leads to
\begin{align}
  & \frac{1}{\tau}\int_\Omega(u^k-u^{k-1})\cdot\phi dx 
	+ \int_\Omega\sum_{i=1}^n D_i(u_0^k)^{1/2}
	\big(\na(u_i^k(u_0^k)^{1/2}) - 3u_i^k\na (u_0^k)^{1/2}\big)
	\cdot\na\phi_i dx \nonumber \\
	&\phantom{xx}{}
	+ \int_\Omega\sum_{i=1}^n D_i u_i^ku_0^k\big(\beta z_i\na\Phi^k+\na W_i\big)
	\cdot\na \phi_i dx,	\label{2.tau3} \\
	& \lambda^2\int_\Omega\na\Phi^k\cdot\na\theta dx 
	= \int_\Omega\bigg(\sum_{i=1}^n z_iu_i^k + f\bigg)\theta dx, \label{2.tau4}
\end{align}
for all $\phi=(\phi_1,\ldots,\phi_n)\in X$ and 
$\theta\in H_D^1(\Omega)$, where $u^k:=u$ and $\Phi^k:=\Phi$.
A density argument shows that we may take $\phi\in H^1_D(\Omega;\R^n)$.

By the trace theorem, $\Phi^k-\Phi^D\in H^1_D(\Omega)$. To show that also 
$u_i^k-u_i(w^D,\Phi^D)\in H_D^1(\Omega;\R^n)$ holds, we observe that
$w^{(\eps)}=w^D$ on $\Gamma_D$ and therefore, $u_0^{(\eps)}=u_0^D$ on $\Gamma_D$
in the sense of traces, where $u_0^D=1-\sum_{i=1}^n u_i^D$ and
$u_i^D:=u_i(w^D,\Phi^D)$. Since
$u_i^{(\eps)}(u_0^{(\eps)})^{1/2}=u_i^D(u_0^D)^{1/2}$ on $\Gamma_D$
and $\na(u_i^{(\eps)}(u_0^{(\eps)})^{1/2})\rightharpoonup \na(u_iu_0^{1/2})$ 
weakly in $L^2(\Omega)$ (see \eqref{2.na}), the
trace theorem implies that $u_iu_0^{1/2}=u_i^D(u_0^D)^{1/2}$ on $\Gamma_D$.

In Lemma \ref{lem.approx}, we have assumed that $w^D\in H^m(\Omega;\R^n)$
since we have taken $w^k-w^D\in X$ as a test function. We may take a sequence
of functions $(w^D_\delta)$ in $H^m(\Omega;\R^n)$ approximating $w^D$ and
then pass to the limit $\delta\to 0$
to achieve the result for $w^D\in H^1(\Omega;\R^n)$.

{\em Step 4: Limit $\tau\to 0$.} Let $u^{(\tau)}(x,t)=u^k(x)$ and 
$\Phi^{(\tau)}(x,t)=\Phi^k(x)$ for $x\in\Omega$ and $t\in((k-1)\tau,k\tau]$,
$k=1,\ldots,N$, be piecewise in time constant functions. At time $t=0$, we set 
$u^{(\tau)}(\cdot,0)=u^0$. We introduce the shift operator
$(\sigma_\tau u^{(\tau)})(\cdot,t)=u^{k-1}$ for $t\in((k-1)\tau,k\tau]$.
Then, in view of \eqref{2.tau3}-\eqref{2.tau4}, $(u^{(\tau)},\Phi^{(\tau)})$ solves
\begin{align}
  & \frac{1}{\tau}\int_\Omega(u^{(\tau)}-\sigma_\tau u^{(\tau)})\cdot\phi dxdt 
	\nonumber \\
	&\phantom{xx}{}+ \int_0^T\int_\Omega\sum_{i=1}^n D_i
	\Big((u_0^{(\tau)})^{1/2}\na\big(u_i^{(\tau)}(u_0^{(\tau)})^{1/2}\big) 
	- 3u_i^{(\tau)}(u_0^{(\tau)})^{1/2}
	\na(u_0^{(\tau)})^{1/2}\Big)\cdot\na\phi_i dxdt \nonumber \\
	&\phantom{xx}{}+ \int_0^T\int_\Omega\sum_{i=1}^n D_i u_i^{(\tau)}u_0^{(\tau)}
	\big(\beta z_i\na\Phi^{(\tau)} + \na W_i\big)\cdot\na\phi_i dxdt = 0, 
	\label{2.tau5} \\
	& \lambda^2\int_0^T\int_\Omega\na\Phi^{(\tau)}\cdot\na\theta dxdt
	= \int_0^T\int_\Omega\bigg(\sum_{i=1}^n z_iu_i^{(\tau)} + f\bigg)\theta dxdt
	\label{2.tau6}
\end{align}
for all piecewise constant functions $\phi_i$, $\theta:(0,T)\to H_D^1(\Omega)$.

Lemma \ref{lem.est} provides the following uniform bounds:
\begin{align}\label{2.uH1}
  \|u_i^{(\tau)}\|_{L^\infty(Q_T)} + \|(u_0^{(\tau)})^{1/2}\|_{L^2(0,T;H^1(\Omega))}
	+ \|u_0^{(\tau)}\|_{L^2(0,T;H^1(\Omega))} &\le C, \\
	\|u_i^{(\tau)}(u_0^{(\tau)})^{1/2}\|_{L^2(0,T;H^1(\Omega))} &\le C, \label{2.uiu0}
\end{align}
where $Q_T=\Omega\times(0,T)$ and $C>0$ is independent of $\tau$.
Moreover, 
$$
  \|\Phi^{(\tau)}\|_{L^2(0,T;H^1(\Omega))}^2 
	= \tau\sum_{k=1}^N \|\Phi^k\|_{H^1(\Omega)}^2
	\le \tau NC \le TC.
$$

We wish to derive a uniform bound for the discrete time derivative of
$(u_i^{(\tau)})$. To this end, we estimate
\begin{align*}
  \frac{1}{\tau}&\bigg|\int_\Omega(u^{(\tau)}-\sigma_\tau u^{(\tau)})
	\cdot\phi dxdt\bigg| 
	\le \int_0^T\sum_{i=1}^n D_i\|u_0^{(\tau)}\|_{L^\infty(\Omega)}^{1/2} \\
	&\phantom{xx}{}\times
	\Big(\|\na(u_i^{(\tau)}(u_0^{(\tau)})^{1/2})\|_{L^2(\Omega)}
	+ 3\|u_i^{(\tau)}\|_{L^\infty(\Omega)}
	\|\na(u_0^{(\tau)})^{1/2}\|_{L^2(\Omega)}
	\Big)\|\na\phi_i\|_{L^2(\Omega)}dt \\
	&\phantom{xx}{}+ \int_0^T\sum_{i=1}^n D_i
	\|u_i^{(\tau)}u_0^{(\tau)}\|_{L^\infty(\Omega)}
	\Big(\beta|z_i|\|\na\Phi^{(\tau)}\|_{L^2(\Omega)}+\|\na W_i\|_{L^2(\Omega)}\Big)
	\|\na\phi_i\|_{L^2(\Omega)}dt \\
	&\le C.
\end{align*}
This holds for all piecewise constant functions $\phi_i:(0,T)\to H_D^1(\Omega)$.
By a density argument, we obtain
\begin{equation}\label{2.utau}
  \tau^{-1}\|u_i^{(\tau)}-\sigma_\tau u_i^{(\tau)}\|_{L^2(0,T;H^1_D(\Omega)')} \le C,
	\quad i=1,\ldots,n.
\end{equation}
Summing these estimates for $i=1,\ldots,n$, we also have
\begin{equation}\label{2.u0tau}
  \tau^{-1}\|u_0^{(\tau)}-\sigma_\tau u_0^{(\tau)}\|_{L^2(0,T;H^1_D(\Omega)')}\le C.
\end{equation}

{}From these estimates, we conclude that, as $\tau\to 0$, up to a subsequence, 
\begin{align*}
  u_i^{(\tau)} \rightharpoonup^* u_i &\quad\mbox{weakly* in }L^\infty(Q_T), \\
	\Phi^{(\tau)}\rightharpoonup \Phi &\quad\mbox{weakly in }L^2(0,T;H^1(\Omega)), \\
	\tau^{-1}(u_i^{(\tau)}-\sigma_\tau u_i^{(\tau)})\rightharpoonup \pa_t u_i
	&\quad\mbox{weakly in }L^2(0,T;H^1_D(\Omega)'),\ i=1,\ldots,n.
\end{align*}
Taking into account \eqref{2.uH1} and \eqref{2.u0tau}, 
we can apply the Aubin-Lions lemma in the version of \cite{DrJu12} to $(u_0^{(\tau)})$
to obtain the existence of a subsequence, which is not relabeled, such that
$u_0^{(\tau)}\to u_0$ strongly in $L^2(Q_T)$,
and this convergence even holds in $L^p(Q_T)$ for $p<\infty$. As a consequence,
\begin{equation}\label{2.u0L2}
  (u_0^{(\tau)})^{1/2}\to u_0^{1/2} \quad\mbox{strongly in }L^p(Q_T),\ p<\infty.
\end{equation}
Thus, by \eqref{2.uH1}, up to a subsequence,
$$
  \na (u_0^{(\tau)})^{1/2}\rightharpoonup \na u_0^{1/2}\quad \mbox{weakly in }L^2(Q_T).
$$

We cannot infer the strong convergence of $(u_i^{(\tau)})$ because of the
degeneracy occurring in estimate \eqref{2.uiu0}. The idea is to employ
the Aubin-Lions lemma in the ``degenerate'' version of \cite{BDPS10,Jue15}
(also see the Appendix in \cite{Jue16}).
In view of \eqref{2.u0L2}, the $L^2(0,T;H^1(\Omega))$ estimates for 
$(u_i^{(\tau)}(u_0^{(\tau)})^{1/2})$ and $((u_0^{(\tau)})^{1/2})$ 
(see \eqref{2.uH1}-\eqref{2.uiu0}), as well as estimate \eqref{2.utau},
there exists a subsequence (not relabeled) such that
\begin{equation}\label{2.conv}
  u_i^{(\tau)}(u_0^{(\tau)})^{1/2} \to u_iu_0^{1/2} \quad\mbox{strongly in }L^2(Q_T).
\end{equation}
Taking into account the uniform bound \eqref{2.uiu0}, we also have
$$
  \na\big(u_i^{(\tau)}(u_0^{(\tau)})^{1/2}\big) \rightharpoonup
	\na(u_iu_0^{1/2}) \quad\mbox{weakly in }L^2(Q_T).
$$
This shows that
\begin{align*}
  (u_0^{(\tau)})^{1/2}\na\big(u_i^{(\tau)}(u_0^{(\tau)})^{1/2}\big) 
	- 3u_i^{(\tau)}(u_0^{(\tau)})^{1/2}
	\na(u_0^{(\tau)})^{1/2}
	\rightharpoonup u_0^{1/2}\na(u_iu_0^{1/2}) - 3u_iu_0^{1/2}\na u_0^{1/2}
\end{align*}
weakly in $L^1(Q_T)$. Furthermore, by \eqref{2.u0L2} and \eqref{2.conv},
$$
  u_i^{(\tau)}u_0^{(\tau)}
  = u_i^{(\tau)}(u_0^{(\tau)})^{1/2}\cdot(u_0^{(\tau)})^{1/2}\to u_iu_0
	\quad\mbox{strongly in }L^2(Q_T).
$$

These convergences allow us to perform the limit $\tau\to 0$ in 
\eqref{2.tau5}-\eqref{2.tau6} to find that $(u_i,\Phi)$ solves
\eqref{1.weak1}-\eqref{1.weak2} for all smooth test functions. By a density
argument, we may take test functions from $L^2(0,T;H^1_D(\Omega))$.
We can show as in Step 3 that the Dirichlet boundary conditions are satisfied, and
the initial condition $u_i(\cdot,0)=u^0_i$ in $\Omega$ follows from arguments
similar as at the end of the proof of Theorem 2 in \cite{Jue15}.


\section{Uniqueness of weak solutions}\label{sec.unique}

We prove Theorem \ref{thm.unique}. For this, we proceed in two steps.

{\em Step 1.} Adding \eqref{1.eq} from $i=1,\ldots,n$ and taking into account
the assumptions $D_i=1$ and $z_i=z$, we find that $u_0=1-\sum_{i=1}^n u_i$ solves
\begin{equation}\label{4.eq}
  \pa_t u_0 = \diver\big(\na u_0 - u_0(1-u_0)(\beta z\na\Phi+\na W)\big), \quad
	-\lambda^2\Delta\Phi = z(1-u_0) + f(x)
\end{equation}
in $\Omega$, $t>0$, 
where $W=\sum_{i=1}^n W_i$, together with the initial conditions 
$u_0(\cdot,0)=1-\sum_{i=1}^n u_i^0$ and boundary conditions
\eqref{1.bc2} and 
$$
  \big(\na u_0 - u_0(1-u_0)(\beta z\na\Phi+\na W)\big)\cdot\nu=0\quad\mbox{on }\Gamma_N,
	\quad u_0=1-\sum_{i=1}^n u_i^D\quad\mbox{on }\Gamma_D.
$$
We show that this problem has a unique weak
solution $(u_0,\Phi)$ in the class of functions 
$\Phi\in L^\infty(0,T;W^{1,q}(\Omega))$.

Let $(u_0,\Phi)$ and $(v_0,\Psi)$ be two weak solutions to \eqref{4.eq} with
the corresponding initial and boundary conditions such that $\Phi$, $\Psi\in
L^\infty(0,T;W^{1,q}(\Omega))$. We take $u_0-v_0$ as a test function in the
weak formulation of the difference of \eqref{4.eq} satisfied by $u_0$ and $v_0$,
respectively. Then
\begin{align}
  \frac12 & \int_\Omega (u_0-v_0)^2(t) dx
	+ \int_0^t\int_\Omega|\na(u_0-v_0)|^2 dxds \nonumber \\
	&= \int_0^t\int_\Omega \Big(u_0(1-u_0)(\beta z\na\Phi+\na W)\big)
	- v_0(1-v_0)(\beta z\na\Psi+\na W)\big)\Big) \nonumber \\
	&\phantom{xx}{}\times\na(u_0-v_0)dxds \nonumber \\
  &= \int_0^t\int_\Omega\big(u_0(1-u_0)-v_0(1-v_0)\big)(\beta z\na\Phi+\na W)
	\cdot\na(u_0-v_0)dxds \nonumber \\
	&\phantom{xx}{}
	+ \beta z\int_0^t\int_\Omega v_0(1-v_0)\na(\Phi-\Psi)\cdot\na(u_0-v_0)dxds 
	\nonumber \\
	&=: I_1 + I_2. \label{4.aux}
\end{align}
The first integral is estimated using the identity 
$u_0(1-u_0)-v_0(1-v_0)=(1-u_0-v_0)(u_0-v_0)$ and 
H\"older's inequality with $1/p+1/q+1/2=1$, where $q>d$ 
(and $2<p<\infty$ if $d\le 2$):
\begin{align*}
  I_1 &\le \|1-u_0-v_0\|_{L^\infty(Q_t)}\|u_0-v_0\|_{L^2(0,t;L^p(\Omega))}
	\|\beta z\na\Phi+\na W\|_{L^\infty(0,t;L^q(\Omega))} \\
	&\phantom{xx}{}\times\|\na(u_0-v_0)\|_{L^2(0,t;L^2(\Omega))} \\
	&\le \frac14\|\na(u_0-v_0)\|_{L^2(Q_t)}^2
	+ C\|u_0-v_0\|_{L^2(0,t;L^p(\Omega))}^2.
\end{align*}
By the Gagliardo-Nirenberg inequality with $\theta=d/2-d/p\in(0,1)$,
\begin{align*}
  \int_0^t\|u_0-v_0\|_{L^p(\Omega)}^2 ds
	&\le C\int_0^t\|u_0-v_0\|_{H^1(\Omega)}^{2\theta}
	\|u_0-v_0\|_{L^2(\Omega)}^{2(1-\theta)}ds \\
	&\le C\int_0^t\big(\|\na(u_0-v_0)\|_{L^2(\Omega)}^{2\theta} 
	+ \|u_0-v_0\|_{L^2(\Omega)}^{2\theta}\big)\|u_0-v_0\|_{L^2(\Omega)}^{2(1-\theta)}ds \\
	&\le \frac14\int_0^t\|\na(u_0-v_0)\|_{L^2(\Omega)}^2ds
	+ C\int_0^t\|u_0-v_0\|_{L^2(\Omega)}^2ds.
\end{align*}
This shows that
$$
  I_1 \le \frac12\|\na(u_0-v_0)\|_{L^2(Q_t)}^2
	+ C\|u_0-v_0\|_{L^2(Q_t)}^2.
$$

For the remaining integral, we employ the following elliptic estimate
$$
  \|\na(\Phi-\Psi)\|_{L^2(\Omega)}\le C\|(1-u_0)-(1-v_0)\|_{L^2(\Omega)}
	= C\|u_0-v_0\|_{L^2(\Omega)},
$$
such that
\begin{align*}
  I_2 &\le \beta|z|\|v_0(1-v_0)\|_{L^\infty(Q_t)}\|\na(\Phi-\Psi)\|_{L^2(Q_t)}
  \|\na(u_0-v_0)\|_{L^2(Q_t)} \\
	&\le C\|u_0-v_0\|_{L^2(Q_t)}\|\na(u_0-v_0)\|_{L^2(Q_t)}
	\le \frac12\|\na(u_0-v_0)\|_{L^2(Q_t)}^2 + \frac{C}{2}\|u_0-v_0\|_{L^2(Q_t)}^2.
\end{align*}
Then, inserting the estimates for $I_1$ and $I_2$ into \eqref{4.aux} leads to
$$
  \frac12\int_\Omega (u_0-v_0)^2(t) dx
	\le C\int_0^t\int_\Omega(u_0-v_0)^2dxds,
$$
and we conclude with Gronwall's lemma that $u_0=v_0$. Consequently, by
the Poisson equation in \eqref{4.eq}, $\Phi=\Psi$.

{\em Step 2.} Next, we show that $u_1,\ldots,u_n$ is the unique weak solution to
\eqref{1.eq}, written in the form
\begin{equation}\label{4.eq2}
  \pa_t u_i = \diver(u_0\na u_i - u_i\na F_i), \quad i=1,\ldots,n,
\end{equation}
where $F_i=u_0+\beta z\Phi+W_i$,
and $(u_0,\Phi)$ is the unique solution to \eqref{4.eq}, together with the
corresponding initial and boundary conditions. 
Since we have assumed that $u_i\in L^2(0,T;H^1(\Omega))$, the formulation
\eqref{1.eq} can be used instead of \eqref{1.weak1}. 
The classical uniqueness proof
requires that $\na F_i\in L^\infty(0,T;L^q(\Omega))$; see the first step of this
proof. To avoid this condition, we use the entropy method of Gajewski
\cite{Gaj94,GaSk04}. 

Let $u=(u_1,\ldots,u_n)$ and $v=(v_1,\ldots,v_n)$ be two weak solutions to
\eqref{4.eq2} with initial and boundary conditions \eqref{1.bc1} and \eqref{1.ic}.
We introduce the semimetric
$$
  d_\eps(u,v) = \int_\Omega\sum_{i=1}^n\bigg(h_\eps(u_i) + h_\eps(v_i)
	- 2h_\eps\bigg(\frac{u_i+v_i}{2}\bigg)\bigg)dx,
$$
where $h_\eps(s)=(s+\eps)(\log(s+\eps)-1)+1$ for $s\ge 0$. 
 The regularization with $\eps>0$ is needed to avoid that expressions like 
$\log(u_i)$ are undefined if $u_i=0$. 
Since $h_\eps$ is convex, we have $h_\eps(u_i)+h_\eps(v_i)-2h_\eps((u_i+v_i)/2)\ge 0$
in $\Omega$ and hence, $d_\eps(u,v)\ge 0$. Now, using \eqref{4.eq2}, we compute, 
similarly as in \cite{ZaJu17},
\begin{align*}
  \frac{d}{dt}d_\eps(u,v) 
	&= \sum_{i=1}^n\bigg\{\bigg\langle\pa_t u_i,
	h'_\eps(u_i)-h'_\eps\bigg(\frac{u_i+v_i}{2}\bigg)\bigg\rangle
	+ \bigg\langle\pa_t v_i,h'_\eps(v_i)-h'_\eps\bigg(\frac{u_i+v_i}{2}\bigg)
	\bigg\rangle\bigg\} \\
	&= -\int_\Omega\sum_{i=1}^n\bigg\{\big(u_0\na u_i-u_i\na F_i)\cdot
	\bigg(h''_\eps(u_i)\na u_i - \frac12 h''_\eps\bigg(\frac{u_i+v_i}{2}\bigg)
	\na(u_i+v_i)\bigg) \\
	&\phantom{xx}{}+ \big(u_0\na v_i-v_i\na F_i)\cdot
	\bigg(h''_\eps(v_i)\na v_i - \frac12 h''_\eps\bigg(\frac{u_i+v_i}{2}\bigg)
	\na(u_i+v_i)\bigg)\bigg\}dx.
\end{align*}
Rearranging these terms, we arrive at
\begin{align*}
  \frac{d}{dt}d_\eps(u,v) 
	&= -4\int_\Omega u_0\sum_{i=1}^n\Big(|\na\sqrt{u_i+\eps}|^2
	+ |\na\sqrt{v_i+\eps}|^2 - 2|\na\sqrt{u_i+v_i+2\eps}|^2\Big)dx \\
	&\phantom{xx}{}- \int_\Omega\sum_{i=1}^n\bigg(\frac{u_i+v_i}{u_i+v_i+2\eps}
	- \frac{u_i}{u_i+\eps}\bigg)\na F_i\cdot\na u_i dx \\
	&\phantom{xx}{}- \int_\Omega\sum_{i=1}^n\bigg(\frac{u_i+v_i}{u_i+v_i+2\eps}
	- \frac{v_i}{v_i+\eps}\bigg)\na F_i\cdot\na v_i dx.
\end{align*}
Lemma 10 in \cite{ZaJu17} shows that the first integral is nonnegative.
Therefore, integrating the above identity in time and observing that
$d_\eps(u(0),v(0))=0$, we obtain
\begin{align*}
  d_\eps(u(t),v(t)) 
	&\le - \int_0^t\int_\Omega\sum_{i=1}^n\bigg(\frac{u_i+v_i}{u_i+v_i+2\eps}
	- \frac{u_i}{u_i+\eps}\bigg)\na F_i\cdot\na u_i dxds \\
	&\phantom{xx}{}- \int_0^t\int_\Omega\sum_{i=1}^n\bigg(\frac{u_i+v_i}{u_i+v_i+2\eps}
	- \frac{v_i}{v_i+\eps}\bigg)\na F_i\cdot\na v_i dxds.
\end{align*}
Arguing as in \cite[Section 6]{ZaJu17}, the dominated convergence theorem shows that
$d_\eps(u(t),v(t))$ $\to 0$ as $\eps\to 0$ (here, we use $\na F_i\in L^2(Q_T)$). 
Then, since a Taylor expansion of $h_\eps$
gives
$$
  d_\eps(u(t),v(v)) \ge \frac18\sum_{i=1}^n\|u_i(t)-v_i(t)\|_{L^2(\Omega)}^2,
$$
we infer that $u_i(t)=v_i(t)$ in $\Omega$ for $t>0$, $i=1,\ldots,n$, which
finishes the proof.


\section{Numerical simulations}\label{sec.num}

We illustrate numerically the behavior of the solutions to \eqref{1.eq}-\eqref{1.poi}
for a specific type of ion channel modeled in \cite{GNE02}. First, our numerical 
scheme is verified by comparing our stationary solutions to the profiles obtained in 
\cite{BSW12}. Second, we explore the large-time behavior of the numerical
solutions.

\subsection{Numerical method}

The equations are discretized in time by an implicit Euler method and in space
by a finite-volume scheme. We suppose that $\Omega=(0,1)$ and impose Dirichlet
boundary conditions.

For the finite volume discretization, the domain is divided into uniform cells of 
size $h>0$. The concentrations and the potential are piecewise constant in each cell
with values $u_{i,m}^k$ and $\Phi_{m}^k$, respectively, where $i=1,\ldots,n$,
$m=1,\ldots,M$, at time $k\triangle t$, $k=1,\ldots,K$. These values are
determined by the following system of nonlinear equations:
\begin{align}
  \label{5.num_u}
  h\frac{u^k_{i,m}-u^{k-1}_{i,m}}{\Delta t} &= J^k_{i,m+1/2} - J^k_{i,m-1/2}, \\
  \label{5.num_phi}
  -\frac{\lambda^2}{h}(\Phi^k_{m+1}-2\Phi^k_m+\Phi^k_{m-1}) 
	&= h \bigg( \sum\limits_{i=1}^{n}z_i u^k_{i,m} + f_m \bigg),
\end{align}
for $i=1,\ldots,n$, $m=1,\ldots,M$, and $k=1,\ldots,K$. The Dirichlet boundary 
conditions are accounted for by setting $\Phi^k_{0}=\Phi^{D}(0)$ and 
$\Phi^k_{M+1}=\Phi^{D}(1)$, and similarly for the concentrations. 
Furthermore, we set $f_m=\frac{1}{h}\int_{(m-1)h}^{mh}f\,dx$, and the fluxes 
$J^k_{i,m\pm 1/2}$ from cell $m$ to cell $m\pm 1$ are given by
\begin{align*}
  J^k_{i,m\pm 1/2}
	&= \pm\frac{D_i}{h}\Big( u^k_{0,m\pm 1/2}(u^k_{i,m\pm 1}-u^k_{i,m}) 
	-u^k_{i,m\pm 1/2}(u^k_{0,m\pm 1}-u^k_{0,m}) \\
  &\phantom{xx}{}+\beta z_iu^k_{i,m\pm 1/2}u^k_{0,m\pm 1/2}
	(\Phi^k_{m\pm 1}-\Phi^k_m) \Big).
\end{align*}
The concentrations at the cell borders are determined by the logarithmic mean
of the cell values:
$$
  u^k_{i,m\pm 1/2} 
	= \begin{cases}\displaystyle
  \frac{u^k_{i,m\pm 1}-u^k_{i,m}}{\log u^k_{i,m\pm 1}-\log u^k_{i,m}} \quad 
	& \text{if }u^k_{i,m\pm 1}>0 \text{ and } u^k_{i,m}>0, \\
  u^k_{i,m} & \text{if }u^k_{i,m\pm 1}=u^k_{i,m}>0, \\
  0 & \text{else}
  \end{cases}
$$
for $i=0,\ldots,n$.
An advantage of this choice is that the fluxes can be reformulated in terms of
the entropy variables
$$
  J^k_{i,m\pm 1/2} = \pm\frac{D_i}{h}u^k_{i,m\pm 1/2}u^k_{0,m\pm 1/2}
	(w^k_{i,m\pm 1}-w^k_{i,m}),
$$
at least if the concentrations are strictly positive. (We do not use this formulation
in the numerical approximation.) The above scheme is implemented using
MATLAB, version R2015a. The nonlinear discrete system 
\eqref{5.num_u}-\eqref{5.num_phi} is solved by a full Newton method in the
variables $u_i^k$ and $\Phi^k$.


\subsection{Simulation of a calcium-selective ion channel}\label{sec.calc}

We consider a model for an L-type calcium channel described in \cite{GNE02} and 
used for numerical simulations also in \cite{BSW12}. We choose a simple geometry, 
where the channel is made of an impermeable cylinder opening up symmetrically 
into two baths, where Dirichlet boundary conditions are prescribed. For the 
simulations, three different types of ions are taken into account: calcium (Ca$^{2+}$, 
$u_1$), sodium (Na$^+$, $u_2$), and chloride (Cl$^-$, $u_3$). The selectivity filter 
of the channel consists in eight confined oxygen ions (O$^{-1/2}$), which contribute 
to the permanent charge density $f=-u_{\text{O}}/2$ as well as to the sum of 
concentrations in the channel, so that $u_0=1-\sum_{i=1}^{3}u_i-u_{\text{O}}$. 
Since these ions are confined, their concentration is assumed to be constant in time.
The concentration profile used in our simulations is a simple piecewise constant 
function, $u_{\text{O}}(x)=0.89$ for $0.45<x<0.55$ and zero else.

In order to obtain results comparable to \cite{BSW12}, we use the same one-dimensional 
approximation of the three-dimensional model that is based on the assumption that the 
longitudinal extension of the considered domain is much larger than the cross section 
of the channel. This leads to the reduced system of equations
\begin{align}\label{5.eq1}
  a(x)\pa_t u_i &= \diver \big(a(x)D_iu_iu_0\na w_i\big), \\ 
	\label{5.eq2}
  -\lambda^2\diver(a(x)\na\Phi) &= a(x)\bigg(\sum_{i=1}^nz_iu_i + f\bigg),
\end{align}
where $a(x)$ is the cross-sectional area of the domain at $x\in (0,1)$. It is given 
by $a(x)=\pi r(x)^2$, where the radius $r(x)$ is determined by the piecewise linear 
function
\begin{equation*}
  r(x)=\begin{cases}
  0.48-x \quad &\text{for }x<0.4\,, \\
  0.08 &\text{for }0.4\le x\le 0.6\,, \\
  x-0.52 &\text{for }x>0.6\,.
\end{cases}
\end{equation*}

For our simulations, we use the parameters given in \cite[Section 5.1, Table 1]{BSW12}.
The initial concentrations are linear functions connecting the Dirichlet boundary
conditions.  The initial potential is then computed from the corresponding Poisson 
equation. The simulations are carried out until the stationary state is reached
approximately, which we determine by computing the $L^2$ error between the solution 
at two consecutive time steps:
$$
  \text{err}_k = \sum_{i=1}^{3}\bigg( \sum_{m=1}^{M} h(u^k_{i,m}-u^{k-1}_{i,m})^2
	\bigg)^{1/2} + \bigg( \sum_{m=1}^{M} h(\Phi^k_{m}-\Phi^{k-1}_{m})^2\bigg)^{1/2}.
$$
The simulation is terminated as soon as $\text{err}_{k}<10^{-13}$. 
We use the time step size $\triangle t=0.001$ and the mesh size $h=0.01$. 

Figure \ref{fig.sol} shows the three ion concentrations and the electric potential
at various time instances. The scaled concentration values are multiplied by
61.5 mol/liter to obtain physical values. 
For small times, there is more sodium than calcium present inside the channel region, 
due to the higher bath and initial concentration of sodium. After some time, 
the sodium inside the channel is replaced by the stronger positively charged calcium. 
For higher initial calcium concentrations, the calcium selectivity of the channel acts 
immediately. The steady-state solution
from our simulation coincides with the stationary profile computed in 
\cite[Figure 5]{BSW12}, which confirms our numerical scheme. The steady state
is reached after 749 time steps, which corresponds to about 23.7 nanoseconds.

\begin{figure}[htb]
\includegraphics[width=175mm]{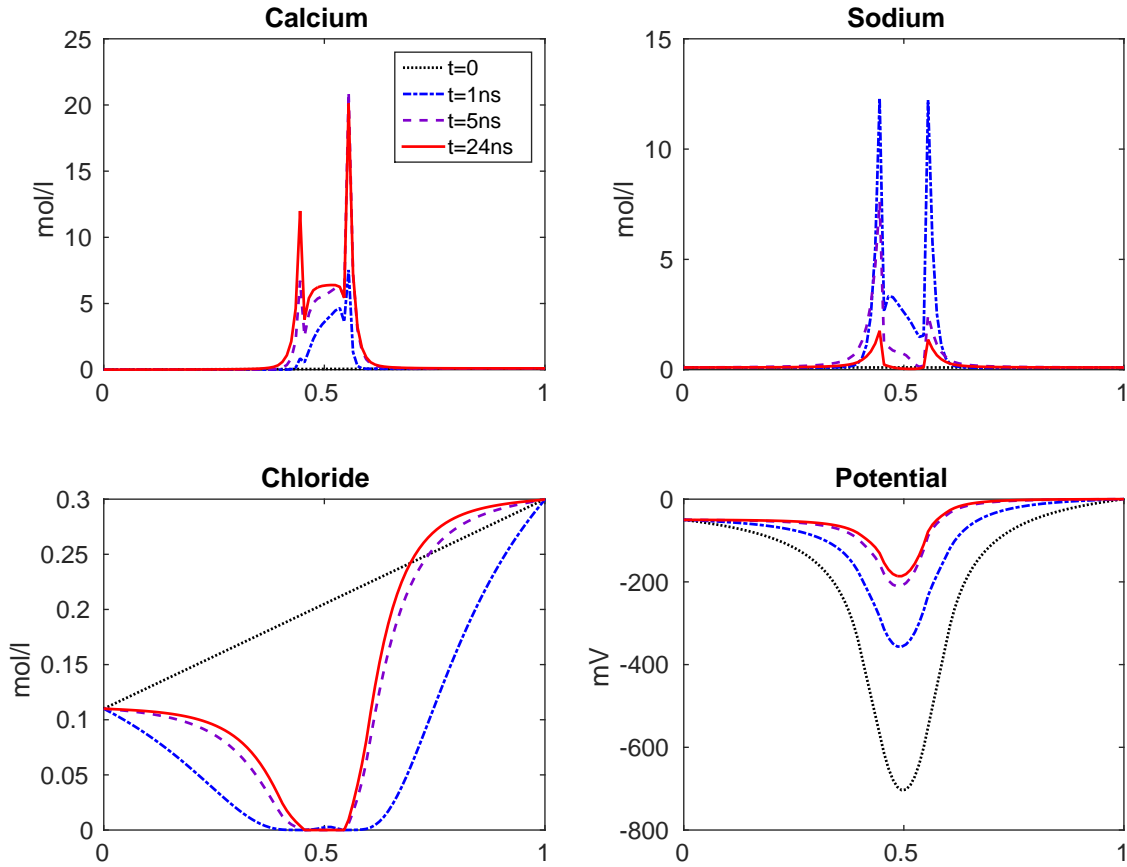}
\caption{Concentrations of calcium, sodium, and chloride ions in mol/l and
electric potential in mV at different times.}
\label{fig.sol}
\end{figure} 


\subsection{Numerical study of the large-time behavior of the solutions}

We investigate numerically the large-time behavior of the solutions and
their decay rates to the equilibrium state. First, we consider the setup
of the previous subsection. Figure \ref{fig.time1} (left) shows the evolution
of the relative entropy \eqref{1.H}, where the boundary data is replaced by
the steady-state solution $(u^\infty,\Phi^\infty)$ (see the previous subsection).  
The right figure displays
the $L^1$ errors $\|u_i^k-u_i^\infty\|_{L^1}$ and $\|\Phi^k-\Phi^\infty\|_{L^1}$
versus the number of time steps $k$. We observe that
the relative entropy converges exponentially fast to the equilibrium state.
By the Csisz\'ar-Kullback inequality (see, e.g., \cite{Jue16} and
references therein), the convergence rate in the $L^1$ norm is expected to
half of that one for the relative entropy, and this is confirmed by 
Figure \ref{fig.time1} (right).

\begin{figure}[htb]
\includegraphics[width=175mm]{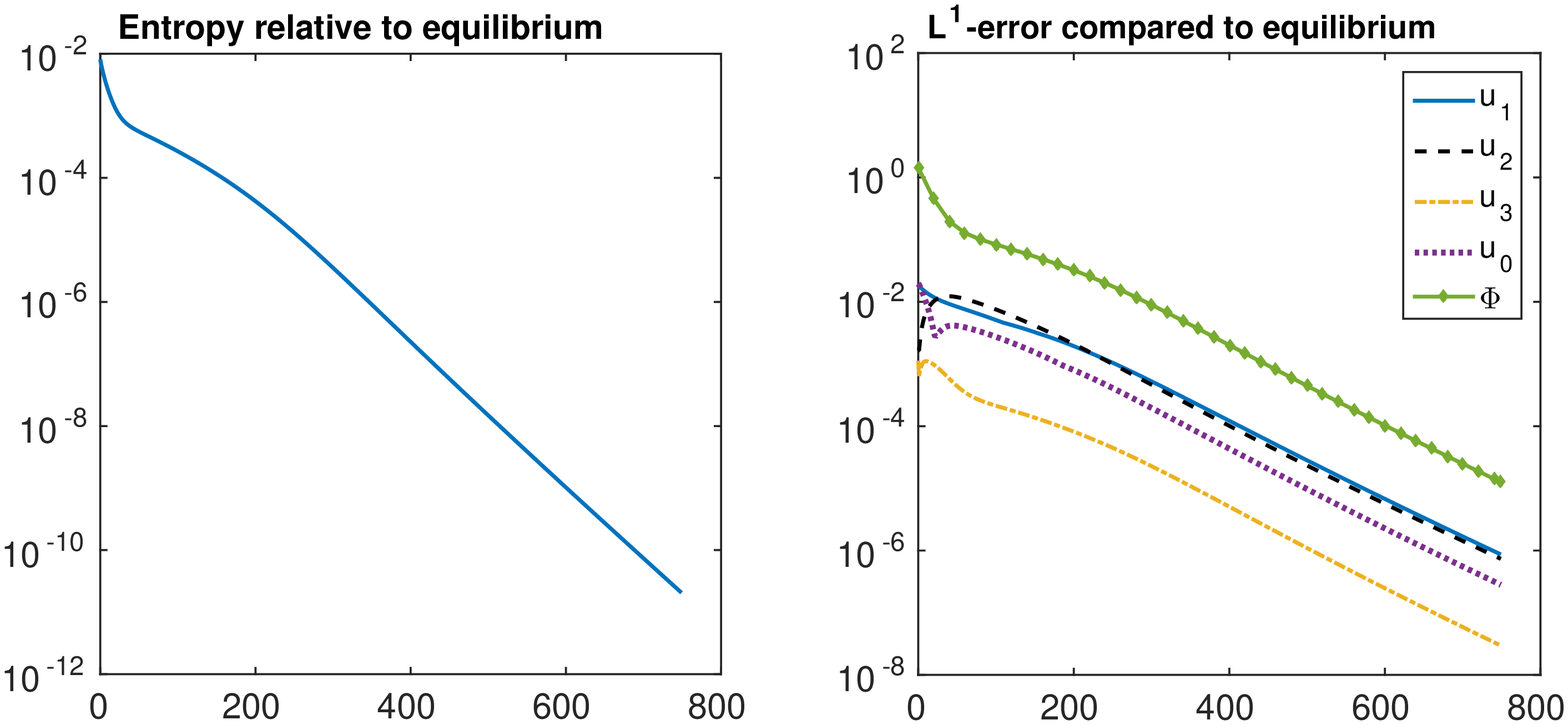}
\caption{Relative entropy (left) and $L^1$ error relative to the steady state
(right) over the number of time steps for the setup of Subsection \ref{sec.calc}.}
\label{fig.time1}
\end{figure} 

Because of the degeneracy at $u_0=0$ in the entropy-production inequality
\eqref{1.epi}, a general proof of exponential convergence rates seems to be
not feasible when the solvent concentration $u_0$ vanishes locally. 
Our second numerical example confirms this statement.
For this, we choose the oxygen concentration
\begin{equation}\label{5.uO}
  u_O(x) = \left\{\begin{array}{ll}
	0.81 & \quad\mbox{for }0.35<x<0.65, \\
	0    & \quad\mbox{else}.
	\end{array}\right.
\end{equation}
All other parameters are kept unchanged.
This choice leads to a solvent concentration $u_0$ that nearly vanishes in a
large part of the computational domain. Consequently, the entropy production 
in \eqref{1.epi} becomes ``small'' and we may expect a rather slow convergence
to equilibrium. Figure \ref{fig.time2} illustrates this behavior. After a short 
initial phase and for the first 20\,000 time steps, the convergence rate is
very small. This comes from the fact that the values of $u_0$ are of the order
$10^{-6}$ in the channel region $x\in[0.4,0.6]$, 
causing the solution to remain nearly unchanged. After about 20\,000 time steps,
the values of $u_0$ increase up to approximately $10^{-3}$ inside the channel
region, which initiates the strong exponential decay to equilibrium.
These results indicate that exponential decay rates cannot be expected
when the solvent concentration vanishes.

\begin{figure}[htb]
\includegraphics[width=175mm]{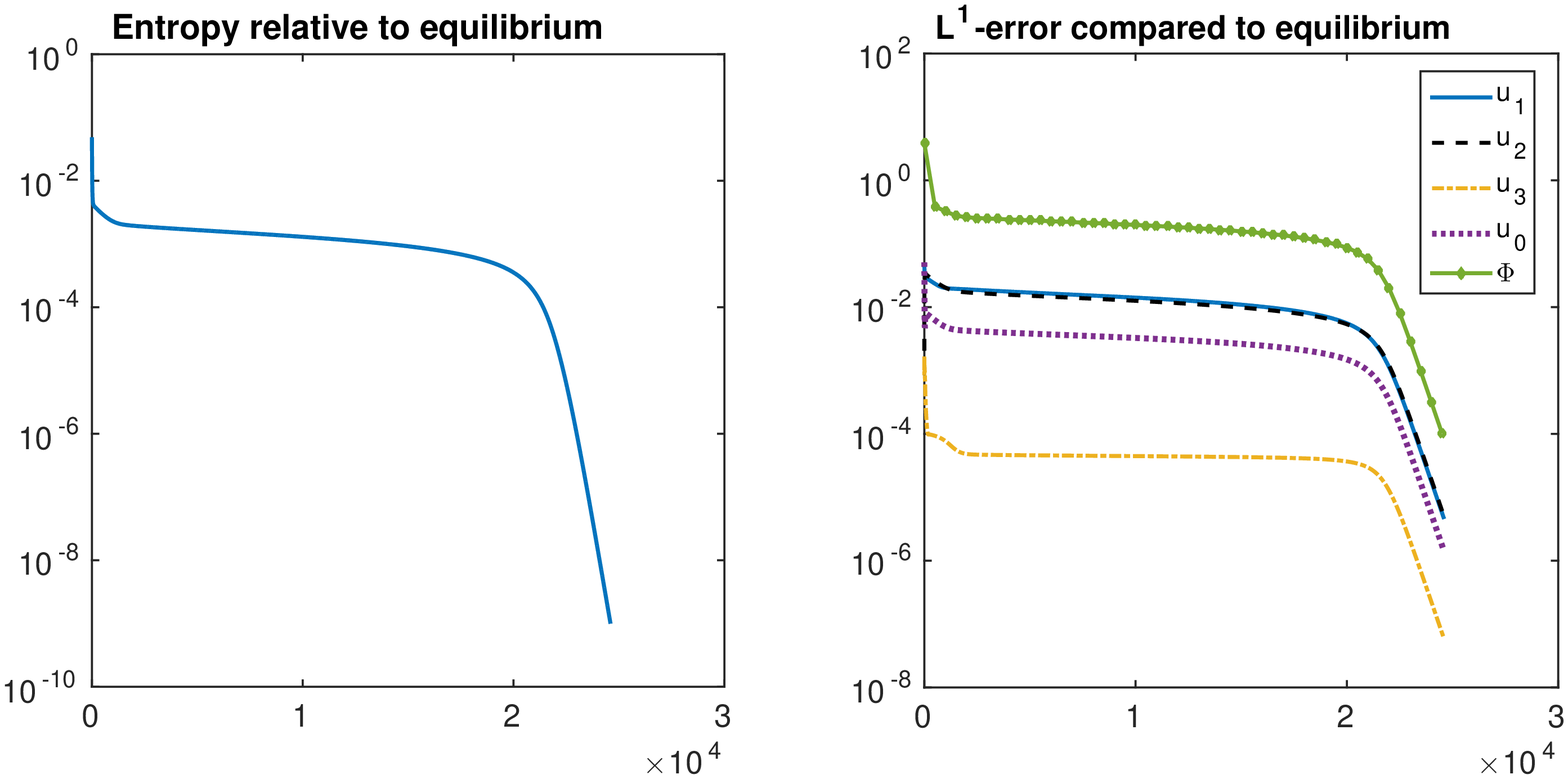}
\caption{Relative entropy (left) and $L^1$ error relative to the steady state
(right) over the number of time steps, computed with the oxygen concentration
\eqref{5.uO}.}
\label{fig.time2}
\end{figure}


\begin{appendix}
\section{Entropy variables}

The appendix is devoted to a (formal) computation of the entropy variables.

\begin{lemma}\label{lem.ev}
Let 
$$
  h(u) = \sum_{i=0}^n\int_{u_i^D}^{u_i}\log\frac{s}{u_i^D}ds
	+ \frac{\beta\lambda^2}{2}|\na(\Phi-\Phi^D)|^2 + \sum_{i=1}^n u_iW_i.
$$
Then
$$
  \frac{\pa h}{\pa u_i} = \log\frac{u_i}{u_0} - \log\frac{u_i^D}{u_0^D}
	+ \beta z_i(\Phi-\Phi^D) + W_i, \quad i=1,\ldots,n.
$$
\end{lemma}

\begin{proof}
It is clear that
$$
  \frac{\pa}{\pa u_i}\bigg(\sum_{i=0}^n\int_{u_i^D}^{u_i}\log\frac{s}{u_i^D}ds
	+ \sum_{i=1}^n u_iW_i\bigg)
	= \log\frac{u_i}{u_i^D} - \log\frac{u_0}{u_0^D} + W_i.
$$
Set $H_{\rm el}(u)=(\beta\lambda^2/2)\int_\Omega|\na\Psi[u]|^2dx$,
where $\Psi[u]=\Phi-\Phi^D$. Recall that $\Phi^D$ solves

$-\lambda^2\Delta\Phi^D=f$ in $\Omega$, $\na\Phi^D\cdot\nu=0$ on $\Gamma_N$.
Then $\Psi[u]$ satisfies $-\lambda^2\Delta\Psi[u]=\sum_{i=1}^n z_iu_i$ in
$\Omega$ together with homogeneous mixed boundary conditions and,  
by the Poisson equation \eqref{1.poi},
$$
  H_{\rm el}(u) = -\frac{\beta\lambda^2}{2}
	\int_\Omega\Delta \Psi[u]\Psi[u] dx
	= \frac{\beta}{2}\int_\Omega\sum_{i=1}^n z_iu_i\Psi[u]dx.
$$
Set $h_{\rm el}(u)=(\beta/2)\sum_{i=1}^n z_iu_i\Psi[u]$.
It remains to show that $\pa h_{\rm el}/\pa u_i=\beta z_i\Psi[u]$. 
For this, we observe that for any (smooth) functions $u=(u_i)$, $v=(v_i)$,
\begin{align}
  \int_\Omega \sum_{i=1}^n z_iu_i\Psi[v]dx
	&= -\lambda^2\int_\Omega\Delta\Psi[u]\Psi[v]dx
	= \lambda^2\int_\Omega\na\Psi[u]\cdot\na\Psi[v]dx \nonumber \\
	&= \int_\Omega \sum_{i=1}^n z_iv_i\Psi[u]dx. \label{2.Psi}
\end{align}
Let $e_i$ be the $i$th unit vector in $\R^n$ and $w$ be a smooth scalar function.
Then, using the linearity of $u\mapsto\Psi[u]$ and \eqref{2.Psi},
\begin{align*}
  \lim_{\eps\to 0} & \frac{1}{\eps}\int_\Omega\bigg(
	h_{\rm el}(u+\eps e_iw)-h_{\rm el}(u)-\eps\beta z_iw\Psi[u]\bigg)dx \\
	&= \frac{\beta}{2}\int_\Omega\bigg(\sum_{j=1}^n z_j\delta_{ij}w\Psi[u] 
	+ \sum_{j=1}^n z_j u_j\Psi[e_iw] - 2z_iw\Psi[u]\bigg) \\
	&= \frac{\beta}{2}\int_\Omega \bigg(z_iw\Psi[u] 
	+ \sum_{j=1}^n z_j\delta_{ij}w\Psi[u] - 2z_iw\Psi[u]\bigg)dx = 0,
\end{align*}
which shows the claim.
\end{proof}

\end{appendix}


\end{document}